\documentclass[11pt]{amsart}

\usepackage[margin=1in]{geometry}
\usepackage{amsmath,amssymb,amsthm,mathtools}
\usepackage[colorlinks=true,linkcolor=blue,citecolor=blue,urlcolor=blue]{hyperref}

\newtheorem{theorem}{Theorem}
\newtheorem{corollary}{Corollary}
\newtheorem{lemma}{Lemma}
\theoremstyle{remark}
\newtheorem*{remark}{Remark}

\newcommand{\cH}{\mathcal H}
\newcommand{\cJ}{\mathcal J}
\newcommand{\cQ}{\mathcal Q}
\newcommand{\cQlin}{\mathcal Q_{\mathrm{lin}}}
\newcommand{\cT}{\mathcal T}
\newcommand{\dd}{\mathrm d}
\newcommand{\rank}{\operatorname{rank}}
\newcommand{\ta}{\tau}
\newcommand{\gam}{\gamma}
\newcommand{\Del}{\Delta}
\newcommand{\DelTwo}{\Delta_2}

\title{An Improved Lower Bound for the Erd\H{o}s--Lov\'asz Cover Number Problem}
\author{Varun Sivashankar}
\address{Department of Mathematics, Princeton University,
Princeton, New Jersey, USA}
\email{varunsiva@princeton.edu}
\date{}

\begin{document}

\begin{abstract}
Let $g(r)$ be the minimum number of edges in an $r$-uniform intersecting
hypergraph with cover number $r$.  Erd\H{o}s and Lov\'asz proved the lower
bound $g(r)\ge 8r/3-3$.  We first give a completely elementary proof that
$g(r)\ge 3r-4$.  We then build on the same approach and apply Kahn's
small-codegree hypergraph edge-colouring theorem to improve this to
$g(r)\ge ((41-\sqrt{19})/12-o(1))r\approx 3.053r$.  In
particular, this shows that $g(r)>3r$ for all sufficiently large $r$,
addressing a question of Erd\H{o}s.
\end{abstract}

\maketitle

\section{Introduction}

A hypergraph is \emph{intersecting} if every two of its edges meet.  A
\emph{cover}, or \emph{transversal}, of a hypergraph $\cH$ is a set of
vertices meeting every edge; its minimum size is denoted $\ta(\cH)$.

Following Erd\H{o}s and Lov\'asz~\cite{EL}, let $g(r)$ be the minimum number
of edges in an $r$-uniform intersecting hypergraph $\cH$ with $\ta(\cH)=r$.
The condition $\ta(\cH)=r$ is the largest possible cover number for an
intersecting $r$-uniform hypergraph, since any edge of $\cH$ is itself a
cover.  Erd\H{o}s and Lov\'asz proved
\[
  g(r)\ge \frac{8r}{3}-3.
\]
Erd\H{o}s later called this one of his three favorite combinatorial problems
and offered \$500 for a proof or disproof of $g(r)\le Cr$.  He also observed
that he and Lov\'asz could not even decide whether
$g(r)<3r$~\cite{ErdosFavorite}.  The problem is listed as Erd\H{o}s
Problem \#21 in Bloom's
collection~\cite{BloomProblem21}.  Kahn answered the
linearity question affirmatively, first up to a logarithmic
factor~\cite{KahnELI}, and then in full~\cite{KahnELII}.  For
small values, Tripathi proved $g(4)=9$ and gave a short proof that
$g(3)=6$~\cite{Tripathi}, while Bar\'at proved $g(5)=13$ and studied
related $2$-intersecting examples~\cite{Barat}.  On the lower-bound side, the
constant $8/3$ remained unchanged for decades; Bar\'at
observed in 2020 that this lower bound had seen no improvement for $45$
years~\cite{Barat}.  Frankl and Tokushige later described determining the
exact value of this function as ``hopelessly difficult''~\cite{FranklTokushigeSurvey}.

A recent line of work studies a vertex-constrained version of the same
problem.  Motivated by Bollob\'as' ``power of many colours'' problem, Alon,
Buci\'c, Christoph, and Krivelevich introduced the variant in which the
hypergraph is required to have a prescribed number of
vertices~\cite{PowerManyColours}; Buci\'c, Jain, and Sivashankar studied this
vertex-constrained problem across the full range~\cite{BJS}.  This
paper returns to the original, unconstrained Erd\H{o}s--Lov\'asz
problem and gives the following explicit lower bound.

\begin{theorem}\label{thm:main}
The following bounds hold.
\begin{enumerate}
\item[(i)] For every positive integer $r$,
\[
  g(r)\ge 3r-4.
\]
\item[(ii)] For every $\varepsilon>0$, there is $r_0$ such that, for every
$r\ge r_0$,
\[
  g(r)\ge \left(\frac{41-\sqrt{19}}{12}-\varepsilon\right)r.
\]
\end{enumerate}
\end{theorem}

Since $(41-\sqrt{19})/12>3$, part~(ii) implies $g(r)>3r$ for all
sufficiently large $r$, crossing the threshold that Erd\H{o}s explicitly
noted was unresolved.

A closely related problem of Meyer~\cite{Meyer} asks for the minimum number
$M(r)$ of edges in a maximal $r$-uniform intersecting hypergraph, where
maximal means that no further $r$-edge can be added while preserving the
intersecting property.  Every maximal $r$-uniform intersecting hypergraph has
cover number exactly $r$: otherwise a cover of size at most $r-1$ could be
extended, using fresh vertices if necessary, to a new $r$-edge meeting every
existing edge, contradicting maximality.  Hence $M(r)\ge g(r)$.  Dow, Drake,
F\"uredi, and Larson proved the lower bound $M(r)\ge 3r$ for
$r\ge 4$~\cite{DDFL}.  Consequently, Theorem~\ref{thm:main}(ii) also gives the
asymptotic improvement
$M(r)\ge ((41-\sqrt{19})/12-o(1))r$.

The proof of part~(i) peels off all vertices of degree at least $4$, leaving
an $r$-uniform intersecting
hypergraph $\cJ$ of maximum degree at most $3$, and then proves a sharp cover
bound for $\cJ$.  This bound is obtained by taking a maximum matching in the
auxiliary triple system formed by the degree-three vertices of $\cJ$ and
counting around the unmatched edges.

The matching argument follows a standard maximum-matching augmentation
philosophy.  After fixing a maximum matching, one studies the link graphs on
the uncovered vertices associated with a matched edge; maximality forces the
relevant link graphs to be pairwise cross-intersecting, since two disjoint
link edges of different types would replace one matched edge by two.  Closely
related link-graph arguments appear in work of Zhang, Zhao, and Lu on
degree-sum conditions for perfect matchings in $3$-uniform
hypergraphs~\cite{ZZL}, especially Lemma~7, where three pairwise
cross-intersecting graphs control local degree sums.  A similar structural
case analysis appears in Frankl and Wang's work on sunflower-free
families~\cite{FranklWangSunflowers}, Lemma~5.1, where two disjoint edges in
one graph force the other two graphs into a small star-like support.

Part~(ii) uses the same high-level approach, but peels at degree $5$ rather
than degree $4$ and then applies Kahn's small-codegree edge-colouring theorem
to the auxiliary $4$-uniform hypergraph formed by the degree-four vertices.
According to its maximum degree, this auxiliary hypergraph supplies either a
large matching or a large linear star; both yield efficient covers.

\section{The Peeling Reduction}

We may restrict attention to simple hypergraphs, since deleting repeated
edge-copies does not change uniformity, the intersecting property, or the cover
number, and can only decrease the number of edges.

Let $\cH$ be an $r$-uniform intersecting hypergraph with $\ta(\cH)=r$, and put
$m=|E(\cH)|$.  Starting from $\cH$, repeatedly choose a vertex of current
degree at least $4$, put it into a set $S$, and delete all current edges
containing it.  Stop when the remaining hypergraph has maximum degree at most
$3$.  Denote the remaining hypergraph by $\cJ$, and put
\[
  k:=|S|,\qquad q:=|E(\cJ)|.
\]
The hypergraph $\cJ$ is still $r$-uniform and intersecting unless it is empty.
If $\cJ$ is empty, then $q=0$ and $\ta(\cJ)=0$.

Each selected vertex deletes at least four previously undeleted edges, so
\begin{equation}\label{eq:peeling-edge-count}
  m-q\ge 4k.
\end{equation}
Moreover, $S$ together with any cover of $\cJ$ covers all of $\cH$.  Hence
\[
  r=\ta(\cH)\le k+\ta(\cJ),
\]
and therefore
\begin{equation}\label{eq:peeling-lower}
  m\ge q+4\bigl(r-\ta(\cJ)\bigr).
\end{equation}

\begin{lemma}\label{lem:q-upper}
Let $\cJ$ be an $r$-uniform intersecting hypergraph with $q$ edges and
$\Del(\cJ)\le 3$.  Then
\[
  q\le 2r+1.
\]
\end{lemma}

\begin{proof}
If $q=0$, the result is immediate.  Otherwise fix an edge $E\in E(\cJ)$.
Since every other edge meets $E$,
\[
  q-1
  \le \sum_{v\in E}\bigl(\dd_{\cJ}(v)-1\bigr).
\]
Each summand is at most $2$, because $\Del(\cJ)\le 3$.  Thus $q-1\le 2r$.
\end{proof}

We can already obtain the Erd\H{o}s--Lov\'asz bound by finding a greedy cover
of $\cJ$ of size at most $q/3+r/6+2/3$.  Indeed, while the current residual
family has $h\ge r+2$ edges, a fixed edge sees total degree at least
$r+h-1\ge 2r+1$, so some vertex has degree $3$; choose it and delete its
three edges.  Once $h\le r+1$, pair the remaining edges.  This gives
\[
  \ta(\cJ)\le \frac q3+\frac r6+\frac23.
\]
Substituting this into~\eqref{eq:peeling-lower} and then using
Lemma~\ref{lem:q-upper} gives
\[
  m\ge \frac{10r}{3}-\frac q3-\frac83
    \ge \frac{8r}{3}-3.
\]
If we had the stronger residual estimate
\[
  4\ta(\cJ)\le q+r+4,
\]
then~\eqref{eq:peeling-lower} would immediately give
\[
  m\ge q+4r-(q+r+4)=3r-4.
\]
Thus Theorem~\ref{thm:main}(i) is reduced to this residual cover bound.  This
is exactly what we prove in Section~3.

\section{The Residual Cover Bound}

\begin{lemma}\label{lem:residual-cover}
Let $\cJ$ be an $r$-uniform intersecting hypergraph with $q$ edges and
$\Del(\cJ)\le 3$.  Then
\[
  4\ta(\cJ)\le q+r+4.
\]
\end{lemma}

\begin{proof}
The proof uses a maximum matching in the auxiliary triple system of
degree-three vertices.  The matched triples will be covered by their defining
vertices, and the unmatched points by pairing; the rest of the proof shows
that this cover is small enough.

For each vertex $v$ of $\cJ$, write
\[
  B_v:=\{A\in E(\cJ):v\in A\}.
\]
We view the edges of $\cJ$ as points.  Since $\Del(\cJ)\le 3$, every block
$B_v$ has size at most $3$; since $\cJ$ is intersecting, every pair of points
is contained in at least one block.

Let $\cT$ be the simple $3$-uniform hypergraph on vertex set $E(\cJ)$ whose
edges are the three-element blocks $B_v$.  Repeated blocks are retained only
once.  This does not affect the matching number, since two equal triples
cannot both appear in a matching.  Let
\[
  M=\{M_1,\ldots,M_t\}
\]
be a maximum matching in $\cT$, and put
\[
  C:=M_1\cup\cdots\cup M_t,\qquad U:=E(\cJ)\setminus C,\qquad u:=|U|.
\]
Then
\begin{equation}\label{eq:q-tu}
  q=3t+u.
\end{equation}
There is no edge of $\cT$ contained in $U$, since such an edge could be added
to $M$.

For each matching triple $M_i$, choose a vertex $w_i$ of $\cJ$ with
$B_{w_i}=M_i$.  These $t$ vertices cover all points in $C$.  Pair the points
of $U$ arbitrarily, leaving one singleton if $u$ is odd.  Each pair is covered
by one common vertex, and the singleton by any one of its vertices.  Thus
$\ta(\cJ)\le t+\lceil u/2\rceil$.

It is enough to prove
\begin{equation}\label{eq:alpha-target}
  r\ge u+t-2.
\end{equation}
Indeed, this would give
\[
  4\ta(\cJ)\le 4t+4\left\lceil\frac u2\right\rceil
  \le 4t+2u+2\le q+r+4.
\]
Put
\[
  \alpha:=r-u-t+2.
\]
We must show $\alpha\ge 0$.  If $u=0$, then $q=3t$, so
Lemma~\ref{lem:q-upper} gives $3t\le 2r+1$, and hence
$\alpha=r-t+2\ge (r+5)/3>0$.  We may therefore assume $u>0$.

For each pair $\{x,y\}\subseteq U$, choose one vertex
$v_{xy}\in x\cap y$, and let
\[
  \mathcal W:=\{v_{xy}:\{x,y\}\in \binom U2\}
\]
be the set of selected witnesses.  The block $B_{v_{xy}}$ is either $\{x,y\}$ or
$\{x,y,z\}$ for some $z\in C$: the third point, if present, cannot lie in
$U$.  The chosen vertices are distinct.  Indeed, suppose the same vertex
were chosen for two different pairs in $U$.  If the two pairs are disjoint,
then this vertex lies in four edges of $\cJ$, contradicting
$\Del(\cJ)\le 3$.  If the two pairs overlap, then their union consists of
three points of $U$, and the block of the chosen vertex is exactly this
three-point set.  This would give an edge of $\cT$ contained in $U$, which is
also impossible.

The selected witnesses therefore contribute exactly
\[
  2\binom u2=u(u-1)
\]
incidences with points of $U$.

Set
\[
  K:=\max\{12,u+4\}.
\]

We will use one local graph claim.  In the application, the graph vertices are
the points of $U$, and the three colours are the points of a fixed matching
triple $M_i$.

\smallskip
\noindent\emph{Claim.}
Let $F$ be a simple graph on $u$ vertices whose edges are coloured with at
most three colours.  Suppose no two vertex-disjoint edges have different
colours.  For each colour $j$, let $V_j$ be the set of vertices incident with
an edge of colour $j$, with $V_j=\varnothing$ for unused colours.  Then
\[
  |V_1|+|V_2|+|V_3|\le K.
\]

\smallskip
\noindent\emph{Proof of the claim.}
First suppose some colour class contains two disjoint edges $e$ and $f$, say
in colour $1$.  Let $W=e\cup f$.  Every edge of colour $2$ or $3$ must meet
both $e$ and $f$, so it has both endpoints in $W$.  If there are no such
edges, we are done.  If the family of edges in colours $2$ and $3$ has no
common vertex, then every edge of colour $1$ also has both endpoints in $W$,
and the support sum is at most $12$.  Otherwise these edges have a common
vertex $c$.  There are at most two possible such edges: they must contain $c$
and meet whichever of $e$ and $f$ does not contain $c$.  Since $F$ is simple,
the supports of colours $2$ and $3$ have total size at most $4$, while
$|V_1|\le u$.
Thus the support sum is at most $u+4$.

We may now assume that every colour class is pairwise intersecting.  Such a
graph is either a star or is contained in a triangle.  If all three supports
have size at most $4$, then the support sum is at most $12$.  Otherwise some
support has size at least $5$, so that colour class is a star with at least
four leaves.  Every edge of every other colour must contain the centre of this
star.  Thus each noncentral vertex belongs to at most one support and the
centre belongs to at most three, giving a support sum at most $u+2$.
\hfill$\square$

\smallskip

Fix $i\in\{1,\ldots,t\}$, and write $M_i=\{a_i,b_i,c_i\}$.  Define a
partially coloured simple graph $F_i$ on vertex set $U$ as follows: for a pair
$\{x,y\}\subseteq U$, put the edge $xy$ into $F_i$ with colour
$z\in M_i$ precisely when $B_{v_{xy}}=\{x,y,z\}$.

The graph $F_i$ has no two vertex-disjoint edges of different colours.  If
$xy$ had colour $z$, $x'y'$ had colour $z'$, the pairs $\{x,y\}$ and
$\{x',y'\}$ were disjoint, and $z\ne z'$, then the two triples
$\{x,y,z\}$ and $\{x',y',z'\}$ would be disjoint from one another and from all
$M_j$ with $j\ne i$.  Replacing $M_i$ by these two triples would enlarge the
matching $M$, a contradiction.  Hence the claim applies to $F_i$.

For $z\in M_i$, let $V_{i,z}$ be the support of colour $z$ in $F_i$.  By
construction,
\[
  V_{i,z}
  =
  \{x\in U:\{x,z\}\text{ is covered by a witness in }\mathcal W\}.
\]
If $S_i$ is the number of distinct pairs between $U$ and $M_i$ covered by
witnesses in $\mathcal W$, then
\[
  S_i=\sum_{z\in M_i}|V_{i,z}|\le K
\]
by the claim.  Since the triples $M_i$ partition $C$, the number $S$ of
distinct pairs between $U$ and $C$ covered by selected witnesses satisfies
\[
  S\le tK.
\]
There are $3tu$ pairs between $U$ and $C$, so the number $R$ not covered by
selected witnesses satisfies
\[
  R\ge 3tu-tK.
\]

Now let $w$ range over the vertices of $\cJ$ not in $\mathcal W$, and put
\[
  a(w):=|B_w\cap U|,\qquad c(w):=|B_w\cap C|.
\]
Every one of the $R$ remaining $U$--$C$ pairs is covered by at least one
vertex not in $\mathcal W$.  Since $a(w)+c(w)\le 3$, we have
$a(w)c(w)\le 2a(w)$ for every $w$.  Hence
\[
  \sum_w a(w)\ge R/2,
\]
where the sum is over vertices not in $\mathcal W$.

The total number of incidences between vertices of $\cJ$ and points of $U$ is
$ur$.  Therefore
\begin{equation}\label{eq:central}
  ur\ge u(u-1)+R/2\ge u(u-1)+\frac{3tu-tK}{2}.
\end{equation}

We now prove $\alpha\ge 0$.  Lemma~\ref{lem:q-upper}, together with
$q=3t+u$ and $r=u+t-2+\alpha$, gives
\begin{equation}\label{eq:t-alpha}
  t\le u-3+2\alpha.
\end{equation}
On the other hand, dividing~\eqref{eq:central} by $u$ and subtracting
$u+t-2$ gives
\begin{equation}\label{eq:alpha-incidence}
  \alpha\ge 1+\frac t2-\frac{Kt}{2u}
  =1-\frac{(K-u)t}{2u}.
\end{equation}
Suppose for contradiction that $\alpha<0$.  Since $\alpha$ is an integer,
\eqref{eq:t-alpha} gives $t\le u-5$.  Using this in
\eqref{eq:alpha-incidence}, and using $K\ge u$, gives
\[
  \alpha\ge 1-\frac{(K-u)(u-5)}{2u}.
\]
The right-hand side is greater than $-1$: if $K=12$, then $1\le u\le 8$ and
$(12-u)(u-5)<4u$, while if $K=u+4$, it equals $-1+10/u$.  Thus
$\alpha>-1$, contradicting the assumption that $\alpha$ is a negative
integer.  Hence $\alpha\ge0$, so~\eqref{eq:alpha-target} holds, and the
lemma follows.
\end{proof}

\begin{remark}
The additive constant in Lemma~\ref{lem:residual-cover} is at most one larger
than the best possible general constant.  Indeed, for odd $n$, let the edges
of $\cJ$ be
$E_1,\ldots,E_n$, and for each pair $\{i,j\}$ let $v_{ij}$ be a vertex lying
in exactly $E_i$ and $E_j$.  Then $\cJ$ is $(n-1)$-uniform, intersecting, and
has maximum degree $2$.  A cover of $\cJ$ is exactly an edge cover of $K_n$,
so $\ta(\cJ)=(n+1)/2$.  Thus, with $q=n$ and $r=n-1$, we have
$4\ta(\cJ)=q+r+3$.  Hence no bound of the form
$4\ta(\cJ)\le q+r+C$ can hold in general with $C<3$.  It is natural to ask
whether the $+4$ in Lemma~\ref{lem:residual-cover} can be replaced by $+3$.
\end{remark}

\section{Improving the Constant via Kahn's Theorem}

We now prove Theorem~\ref{thm:main}(ii).  To cross the $3r$ threshold raised
by Erd\H{o}s, we repeat the peeling argument with threshold $5$ and use
Kahn's edge-colouring theorem together with a star alternative in an auxiliary
$4$-uniform hypergraph.  In this section, auxiliary
hypergraphs are allowed to have repeated edge-copies.  For such a hypergraph
$G$, let $\rank(G)$ be the maximum size of an edge-copy, let $\Del(G)$ be the
maximum vertex degree, and let $\DelTwo(G)$ be the maximum codegree of a pair
of distinct vertices; degrees and codegrees are counted with multiplicity.  The
list chromatic index $\chi'_\ell(G)$ is defined in the usual way: every
edge-copy receives a colour from its list, and each colour class is required
to be a matching.

The near-regular chromatic-index theorem underlying this line of work is due
to Pippenger and Spencer~\cite{PippengerSpencer}.  We use the following
small-codegree list-colouring form, stated as Theorem~3.1 in
Kang--Kelly--K\"uhn--Methuku--Osthus~\cite{KKMO}, with attribution to
Kahn~\cite{KahnList}.  Their convention allows repeated edge-copies.

\begin{theorem}[Kahn's small-codegree list edge-colouring theorem]
\label{thm:kahn}
For every integer $s\ge 2$ and every $\gam>0$, there exist $\delta>0$ and
$D_0$ such that the following holds.  If $D\ge D_0$ and $G$ is a hypergraph
with
\[
  \rank(G)\le s,\qquad \Del(G)\le D,\qquad \DelTwo(G)\le \delta D,
\]
then
\[
  \chi'_\ell(G)\le (1+\gam)D.
\]
\end{theorem}

We will use only the following rank-four consequence.

\begin{corollary}\label{cor:kahn-rank-four}
For every $\eta>0$, there is $D_0$ such that the following holds.  Let $G$ be
a $4$-uniform hypergraph with $\Del(G)=D\ge D_0$ and
$\DelTwo(G)\le 1$.  Then
\[
  \chi'(G)\le (1+\eta)D.
\]
Consequently,
\[
  \nu(G)\ge \frac{|E(G)|}{(1+\eta)D}.
\]
\end{corollary}

\begin{proof}
Apply Theorem~\ref{thm:kahn} with $s=4$ and $\gam=\eta$, and let $\delta$ and
$D_0'$ be the resulting threshold.  Increase $D_0'$, if necessary, so that
$1\le\delta D$ whenever $D\ge D_0'$.
The theorem gives
\[
  \chi'(G)\le \chi'_\ell(G)\le (1+\eta)D.
\]
The largest colour class in a proper edge-colouring is a matching of size at
least
\[
  \frac{|E(G)|}{\chi'(G)}
  \ge \frac{|E(G)|}{(1+\eta)D}.
\]
\end{proof}

\begin{proof}[Proof of Theorem~\ref{thm:main}(ii)]
Put
\[
  \beta_0:=\frac{4+\sqrt{19}}{3},
  \qquad
  c:=\frac{41-\sqrt{19}}{12}.
\]
Fix $\varepsilon>0$, and choose $\eta$ so that
\[
  0<\eta\le
  \min\left\{1,\left(\frac{\varepsilon}{2c}\right)^2\right\}.
\]
We prove that $m\ge (c-\varepsilon)r$ for all sufficiently large $r$.

Let $\cH$ be an $r$-uniform intersecting hypergraph with $\ta(\cH)=r$ and
$m=|E(\cH)|$.  Repeating the peeling argument with threshold $5$, peel
vertices of current degree at least $5$.  Let $S$ be the set of peeled
vertices, let $k:=|S|$, and let $\cJ$ be the remaining hypergraph.  Put
$q:=|E(\cJ)|$.  Then $\Del(\cJ)\le 4$, and the same covering argument gives
\begin{equation}\label{eq:five-peeling}
  m\ge q+5\bigl(r-\ta(\cJ)\bigr).
\end{equation}
Also, if $q>0$, then fixing one edge of $\cJ$ and using $\Del(\cJ)\le 4$
gives
\begin{equation}\label{eq:q-upper-four}
  q\le 3r+1.
\end{equation}
The same inequality is trivial when $q=0$.

We first prove a cover bound that reuses Lemma~\ref{lem:residual-cover}.  Let
$\cQ$ be the auxiliary $4$-uniform hypergraph on vertex set $E(\cJ)$ whose
edge-copies are the four-element blocks $B_v$ coming from vertices $v$ of
$\cJ$ of degree $4$.  Let $M$ be a maximal matching in $\cQ$, say
$|M|=a$, and let $U$ be the set of edges of $\cJ$ not covered by the matched
auxiliary edges.  The $a$ corresponding vertices of $\cJ$ cover the
$4a$ edges outside $U$.  Moreover, the subhypergraph of $\cJ$ induced by $U$
has maximum degree at most $3$, since a degree-four vertex inside $U$ would
give an auxiliary edge disjoint from $M$.  Lemma~\ref{lem:residual-cover}
therefore gives
\[
  \ta(\cJ)\le a+\frac{q-4a+r+4}{4}
  =\frac{q+r+4}{4}.
\]
Substituting this into~\eqref{eq:five-peeling} gives
\begin{equation}\label{eq:first-kahn-bound}
  m\ge \frac{15r}{4}-\frac q4-5.
\end{equation}
Since $c=15/4-\beta_0/4$, if $q\le\beta_0r$, then
~\eqref{eq:first-kahn-bound} gives $m\ge cr-5$, which is at least
$(c-\varepsilon)r$ for all sufficiently large $r$.  We may therefore assume
$q>\beta_0r$.  In particular, $q\to\infty$ and $q=\Theta(r)$,
by~\eqref{eq:q-upper-four}.

It remains to prove a second bound, useful when $q$ is large.  Let $x_i$ be
the number of vertices of degree exactly $i$ in $\cJ$, for $1\le i\le 4$.
Then
\begin{equation}\label{eq:degree-four-incidences}
  x_1+2x_2+3x_3+4x_4=qr.
\end{equation}
For distinct $A,B\in E(\cJ)$, let $\lambda_{AB}$ be the number of
degree-four vertices of $\cJ$ lying in $A\cap B$, and define
\[
  X:=\sum_{\{A,B\}\subseteq E(\cJ)}(\lambda_{AB}-1)_+.
\]
We claim that $\cQ$ has a $4$-uniform subhypergraph $\cQlin$ with
$\DelTwo(\cQlin)\le 1$ and
\begin{equation}\label{eq:rank-four-linear-count}
  |E(\cQlin)|\ge \binom q2-\frac{5qr}{4}.
\end{equation}
Indeed, let $z$ be the number of pairs $\{A,B\}$ with $\lambda_{AB}=0$.
For each such pair choose a vertex in $A\cap B$.  This chosen vertex has
degree $2$ or $3$: it cannot have degree $1$, since it lies in both $A$ and
$B$, and it cannot have degree $4$, by the definition of $z$.  A degree-two
vertex lies in one pair of edges of $\cJ$, while a degree-three vertex lies in
three pairs.  Thus
\[
  z\le x_2+3x_3.
\]
Every degree-four vertex of $\cJ$ lies in six pairs of edges, and hence
contributes $6$ to the sum of the pair-codegrees.  Since $\lambda_{AB}>0$
for exactly $\binom q2-z$ pairs,
\[
  X
  =6x_4-\left(\binom q2-z\right)
  \le 6x_4-\binom q2+x_2+3x_3.
\]
Hence
\[
  x_4-X\ge \binom q2-x_2-3x_3-5x_4
  \ge \binom q2-\frac{5qr}{4},
\]
where the last inequality follows from~\eqref{eq:degree-four-incidences}.
Starting with $\cQ$, whenever a pair $\{A,B\}$ of auxiliary vertices has
current codegree $d\ge2$, delete one entire auxiliary edge-copy $B_v$
containing $\{A,B\}$.  The chosen pair's excess drops from $d-1$ to $d-2$,
and no other pair-excess can increase.  Thus the total excess drops by at
least $1$ per deletion.  Initially it is $X$, so at most $X$ edge-copies are
deleted.  At termination, the remaining subhypergraph $\cQlin$ has
$\DelTwo(\cQlin)\le1$ and at least $x_4-X$ edge-copies.  This
proves~\eqref{eq:rank-four-linear-count}.

Write
\[
  e:=|E(\cQlin)|,\qquad D:=\Del(\cQlin).
\]
Since $q>\beta_0r$ and $\beta_0>5/2$,
~\eqref{eq:rank-four-linear-count} gives $e=\Omega(r^2)$.
Consequently, $D\ge 4e/q=\Omega(r)$.  For all sufficiently large $r$,
Corollary~\ref{cor:kahn-rank-four} therefore yields a matching in $\cQlin$
of size
\[
  a\ge \frac{e}{(1+\eta)D}.
\]
The corresponding $a$ vertices of $\cJ$ cover $4a$ distinct edges; pairing
the remaining edges gives
\[
  \ta(\cJ)
  \le a+\left\lceil\frac{q-4a}{2}\right\rceil
  \le \frac q2-\frac{e}{(1+\eta)D}+1.
\]

The above bound is strong when $D$ is small.  If $D$ is large, we can bound
the cover number in a different way.  Choose a vertex of $\cQlin$ of degree
$D$.  Since $\DelTwo(\cQlin)\le1$, the $D$ auxiliary edges through it have no
other vertices in common, and hence their union has size $1+3D$.  The
corresponding $D$ vertices of $\cJ$ cover those $1+3D$ edges.  Pairing the
remaining edges gives
\[
  \ta(\cJ)
  \le D+\left\lceil\frac{q-1-3D}{2}\right\rceil
  \le \frac q2-\frac D2.
\]
Combining the two covers and using the geometric mean,
\begin{equation}\label{eq:matching-star-cover}
  \ta(\cJ)
  \le \frac q2-
  \max\left\{\frac D2,\frac{e}{(1+\eta)D}\right\}+1
  \le \frac q2-\sqrt{\frac{e}{2(1+\eta)}}+1.
\end{equation}
Substituting~\eqref{eq:rank-four-linear-count} into
~\eqref{eq:matching-star-cover}, and then using~\eqref{eq:five-peeling},
gives
\[
  m\ge
  5r-\frac{3q}{2}
  +5\sqrt{\frac{1}{2(1+\eta)}
  \left(\binom q2-\frac{5qr}{4}\right)}-5.
\]

To finish, write $\beta=q/r$ and define
\[
  F(\beta):=
  5-\frac{3\beta}{2}
  +5\sqrt{\frac{2\beta^2-5\beta}{8}}.
\]
For all sufficiently large $r$, we have
$\beta/(4r)\le(2\beta^2-5\beta)/8$, since $\beta>\beta_0>5/2$.
Using $\sqrt{x-y}\ge\sqrt{x}-\sqrt y$ for $x\ge y\ge0$, first with
$x=(2\beta^2-5\beta)/8$ and $y=\beta/(4r)$, and then in
\[
  \frac1{\sqrt{1+\eta}}
  =\sqrt{1-\frac{\eta}{1+\eta}}\ge1-\sqrt\eta,
\]
we obtain, since $\beta\le3+1/r\le4$,
\[
\begin{aligned}
  &\sqrt{\frac{1}{2(1+\eta)}
  \left(\binom q2-\frac{5qr}{4}\right)}\\
  &\qquad\ge
  r(1-\sqrt\eta)\sqrt{\frac{2\beta^2-5\beta}{8}}
  -\sqrt r.
\end{aligned}
\]
Moreover, $5-3\beta/2\ge0$ for all sufficiently large $r$.  Hence
\[
  m\ge r(1-\sqrt\eta)F(\beta)-5\sqrt r-5.
\]

A direct differentiation shows that $F$ is increasing on $(5/2,\infty)$:
after moving the negative term in $F'(\beta)$ to the other side and
squaring, the claim is equivalent to
$256\beta^2-640\beta+625>0$, whose leading coefficient is positive and
whose discriminant is negative.  Also,
$3\beta_0^2-8\beta_0-1=0$ and $\beta_0>1$, so
\[
  F(\beta)\ge F(\beta_0)
  =\frac{15}{4}-\frac{\beta_0}{4}=c.
\]
Consequently,
\[
  m\ge (1-\sqrt\eta)cr-5\sqrt r-5.
\]
Our choice of $\eta$ gives $c\sqrt\eta\le\varepsilon/2$, and for all
sufficiently large $r$ we have $5\sqrt r+5\le\varepsilon r/2$.
Thus $m\ge(c-\varepsilon)r$, as required.
\end{proof}

\section*{Acknowledgements}

We thank Noga Alon and Matija Buci\'c for helpful comments.
The proof was discovered with the help of ChatGPT 5.5 Pro.
Theorem~1 was formalized in Lean with the help of Harmonic's Aristotle:
part~(i) was formalized in full, and part~(ii) was formalized conditional on
Kahn's theorem.

\end{document}